\documentclass[11pt]{amsart}
\usepackage{amsmath}
\usepackage{amssymb,amscd}
\usepackage{epsfig}
\usepackage{graphicx}
\usepackage[colorlinks=true, urlcolor=blue,bookmarks=true,bookmarksopen=true, citecolor=blue,hypertex]{hyperref}
\numberwithin{equation}{section}

\newtheorem{defn}{Definition}[section]
\newtheorem{theorem}{Theorem}[section]

\newtheorem{corollary}[theorem]{Corollary}

\newtheorem{lemma}[theorem]{Lemma}
\newtheorem{prop}[theorem]{Proposition}
\newtheorem{remark}[theorem]{Remark}

\theoremstyle{definition}
\newtheorem{obser}{Observation}
\newtheorem{property}{Property}

\newtheorem{example}[theorem]{Example}

\def \begineq{\begin{equation}}
\def \endeq{\end{equation}}

\def \bb{\mathbb}

\def \QQ{{\bb{Q}}}
\def \RR{{\bb{R}}}
\def \TT{{\bb{T}}}

\def \ZZ{{\bb{Z}}}

\def \({\left(}
\def \){\right)}
\def \<{\langle}
\def \>{\rangle}
\def \bar{\overline}

\begin{document}

\title{ Some small orbifolds over polytopes }
 %MSC classification
 %53C15; 53D20

%\author[M. Poddar]{Mainak Poddar}

%\address{Theoretical Statistics and Mathematics Unit, Indian
%Statistical Institute, 203 B. T. Road, Kolkata 700108, India}

%\email{mainak@isical.ac.in}

\author[S. Sarkar]{Soumen Sarkar}

\address{\it{Theoretical Statistics and Mathematics Unit, Indian
Statistical Institute, 203 B. T. Road, Kolkata 700108, India}}

\email{soumens$_-$r@isical.ac.in}



\abstract We introduce some compact orbifolds on which there is a certain finite group action having a simple
convex polytope as the orbit space. We compute the orbifold fundamental group and homology groups of these
orbifolds. We compute the cohomology rings of these orbifolds when the dimension of the orbifold is even.
These orbifolds are intimately related to the notion of small cover.
\endabstract

\maketitle
\section{Introduction}
An $n$-dimensional simple polytope is a convex polytope in $\RR^n$ where exactly $n$ bounding hyperplanes meet at 
each vertex. The codimension one faces of simple polytope are called facets.
In this article we introduce some $n$-dimensional orbifolds on which there is a natural
$ \ZZ_{2}^{n-1} $ action having a simple polytope as the orbit space, where $\ZZ_2 = \ZZ/2\ZZ $. We call
these orbifolds small orbifolds. The fixed points of $ \ZZ_2^{n-1} $ action on an $n$-dimensional small
orbifold is homeomorphic to the $1$-skeleton of the polytope.
The small orbifolds are closely related to the notion of small cover.
A small cover of dimension $ n $ is an $n$-dimensional smooth manifold endowed with a natural
action of $ \ZZ_{2}^{n} $ having a simple $n$-polytope as the orbit space, see \cite{DJ}. The fixed point set of
$ \ZZ_2^n $ action on a small cover correspond bijectively to the set of vertices of polytope.

In section \ref{def} we give the precise definition of small orbifold and give two examples. We show
the smoothness of small orbifold. In section \ref{orbf} we calculate the orbifold fundamental group of small
orbifolds. We show that the universal orbifold cover of $n$-dimensional ($n > 2$) small orbifold is
diffeomorphic to $\RR^n$. Theorem \ref{isos} shows that the space $\mathcal{Z}$, constructed
in Lemma 4.4 of \cite{DJ}, is diffeomorphic to $\RR^n$ if there is an s-characteristic function (definition \ref{defs})
of simple $n$-polytope. In section \ref{hom} we construct a $CW$-complex structure on small
orbifold. We compute the singular homology groups of small orbifold with integer coefficients.
We establish a relation between the modulo $2$ Betti numbers of a small orbifold and the $h$-$vector$
of the polytope. In section \ref{crto} we compute the singular cohomology groups and the cohomology
ring of even dimensional small orbifold. In the last section we discuss the toric version of small orbifold.
All points of the quotient space is smooth except at finite points. Though the quotient space is not an orbifold
(when $n >2$), we compute the singular homology groups of these spaces.

\section{Definition and orbifold structure}\label{def}
Let $ P $ be a simple polytope of dimension $ n $. %Through out this article $ P $ represents a simple polytope.
Let $\mathcal{F}(P) = \{ F_i, ~ i= 1, 2, \ldots, m \} $ be the set of facets of $P$.
Let $V(P)$ be the set of vertices of $ P $. We denote the underlying additive group of the vector space
$\mathbb{F}_{2}^{n-1} $ by $\ZZ_2^{n-1}$.

\begin{defn}\label{defs}
A function $\vartheta : \mathcal{F}(P) \to \mathbb{Z}_{2}^{n-1}$ is called an s-characteristic function of the
polytope $P$ if the facets $F_{i_{1}}, F_{i_{2}}, \ldots , F_{i_{n}}$ intersect at a vertex of $P$ then the set
$$\{ \vartheta_{i_{1}}, ~\vartheta_{i_{2}}, ~\ldots, ~ \vartheta_{i_{k-1}}, ~ \widehat \vartheta_{i_{k}},
~ \vartheta_{i_{k+1}}, ~ \ldots, ~ \vartheta_{i_{n}} \},$$
where $\vartheta_i := \vartheta(F_i)$, is a basis of $ \mathbb{F}_{2}^{n-1} $ over
$\mathbb{F}_{2}$ for each $ k $, $ 1 \leq k \leq n $.
We call the pair $(P, \vartheta)$ an s-characteristic pair.
\end{defn}

Here the symbol $~\widehat{}~$ represents the omission of corresponding entry.
We give examples of s-characteristic function in \ref{eg01} and \ref{eg02}.
Now we give the constructive definition of small orbifold using the s-characteristic pair $(P, \vartheta)$.

Let $ F $ be a face of the simple polytope $ P $ of codimension $k \geq 1$.
Then $$ F = F_{i_1} \cap F_{i_2} \cap \ldots \cap F_{i_k},$$ where $ F_{i_j} \in \mathcal{F}(P)$
containing $ F $. Let $ G_{F} $ be the subspace of $ \mathbb{F}_{2}^{n-1} $ spanned
by $ \vartheta_{i_1}, \vartheta_{i_2}, \ldots ,$ $ \vartheta_{i_k} $. Without any confusion we denote
the underlying additive group of the subspace $ G_F $ by $ G_F $. By the definition of $ \vartheta $,
$ G_{v} = \ZZ_{2}^{n-1} $ for each $ v \in V(P) $. So the s-characteristic function $\vartheta$ determines
a unique subgroup of $\ZZ_2^{n-1}$ associated to each face of the polytope $ P $. Note that if $k < n$ then
$G_F \cong \ZZ_2^k$. The subgroup $G_F$ of $ \mathbb{Z}_{2}^{n-1} $ is a direct summand.

Each point $ p $ of $ P $ belongs to relative interior of a unique face $ F(p) $ of $ P $.
Define an equivalence relations $\sim$ on $ \mathbb{Z}_{2}^{n-1} \times P $ by
\begin{equation}\label{idenso}
(t,p) \sim (s,q) ~ \mbox{if} ~ p = q ~ \mbox{and} ~ s-t \in G_{F(p)}
\end{equation}

Let $ X(P, \vartheta) = ( \mathbb{Z}_{2}^{n-1} \times P )/\sim $ be the quotient space. Whenever there is
no ambiguity we denote $X(P, \vartheta)$ by $X$.
Then $ X $ is a $ \ZZ_{2}^{n-1} $-space with the orbit map
\begin{equation}
 \pi: X \to P ~\mbox{defined by} ~ \pi([t,p]^{\sim}) = p.
\end{equation}
Let $ \hat{\pi} :\mathbb{Z}_{2}^{n-1} \times P \to X $
be the quotient map. Let $B^n$ be the open ball of radius $1$ in $\RR^n$.

We claim that the space $ X $ has a smooth orbifold structure.
To prove our claim we construct a smooth orbifold atlas. We show that for each vertex $ v $ of $ P $ 
there exists an orbifold chart $(B^n, \ZZ_2, \phi_v)$ of $ X_v $ where $\phi_v (B^n)$ is an open subset $X_v$
of $ X $ and $$\{ X_v : v \in V(P)\} $$ cover $ X $. To show the compatibility of these charts as $v$ varies over $V(P)$,
we introduce some additional orbifold charts to make this collection an orbifold atlas.

Let $v \in V(P)$ and $ U_v $ be the open subset of $ P $ obtained by deleting all faces of $P$
not containing $ v $. Let $$ X_v := \pi^{-1}(U_v) = (\mathbb{Z}_{2}^{n-1} \times U_v)/\sim .$$
The subset $ U_v $ is diffeomorphic as manifold with corners to
\begin{equation}
 B_1^{n} = \{x= (x_1, x_2, \ldots, x_n) \in \RR_{\geq 0}^{n} :  \Sigma^{^{n}}_{_{1}} x_j < 1\}.
\end{equation}
Let $ f_v : B_1^n \to U_v $ be the diffeomorphism. Let the facets
$$ \{x_1 = 0\} \cap B_1^{n},~ \{x_2 =0 \}\cap B_1^{n}, ~ \ldots, ~\{ x_n = 0\} \cap B_1^{n} $$
of $B^n_1$ map to the facets $ F_{i_{1}}, F_{i_{2}}, \ldots, F_{i_{n}} $ of $U_v$
respectively under the diffeomorphism $ f_v $. Then $F_{i_1}\cap F_{i_2} \cap \ldots \cap F_{i_n} =v $.
Define an equivalence relation $\sim_{0}$ on $ \mathbb{Z}_{2}^{n-1} \times B_1^{n} $ by
\begin{equation}\label{eq1}
(t, x) \sim_{0} (s,y) ~ \mbox{if} ~ x=y ~ \mbox{and} ~ s-t \in G_{F(f_{v}(x))}
\end{equation}
Let $Y_0 = (\mathbb{Z}_{2}^{n-1} \times B_1^{n}) / \sim_0 $ be the quotient space with the
orbit map $\pi_0 : Y_0 \to B_1^n$. Let $\hat \pi_0 : \mathbb{Z}_{2}^{n-1} \times B_1^{n} \to Y_{0} $
be the quotient map. The diffeomorphism 
$ {id\times f_v}$ % :\mathbb{Z}_{2}^{n-1} \times B_1^n \to \mathbb{Z}_{2}^{n-1} \times U_v $$
 descends to the following commutative diagram.
\begin{equation}
\begin{CD}
\mathbb{Z}_{2}^{n-1} \times B_1^n @>{id\times f_v}>> \mathbb{Z}_{2}^{n-1} \times U_v \\
@V\hat \pi_0 VV  @V\hat \pi_v VV \\
 Y_0 @>\hat f_v >> X_v
\end{CD}
\end{equation}
Here $ \hat \pi_v $ is the map $\hat \pi$ restricted to $\mathbb{Z}_{2}^{n-1} \times U_v$. It is easy to 
observe that the map $ \hat f_v $ is a bijection. Since the maps $ \hat \pi_v $ and $\hat \pi_0$ are continuous
and the map $id\times f_v$ is a diffeomorphism, the map $\hat f_v $ is a homeomorphism.

Let $ a \in [0,1) $ and $ H_a $ be the hyperplane $\{ \Sigma^{^{n}}_{_{1}} x_j = a  \} $ in  $ \RR^{n} $.
Then $ P_0 = H_0 \cap B_1^n $ is the origin of $\RR^n$ and $ P_a = H_a \cap B_1^{n} $ is an $ (n-1) $-simplex for each
$ a \in (0,1) $. When $a \in (0,1)$, the facets of $ P_a $ are $$ \{F_{a_j} := \{x_j = 0\} \cap P_a ;~ j=1,2, \ldots, n\}.$$
The map 
\begin{equation}
 \vartheta_{a}: \{F_{a_j}: j=1, \ldots, n\} \to \ZZ_2^{n-1}~ \mbox{defined by} ~ \vartheta_{a}({F_{a_j}}) = \vartheta_{i_j}
\end{equation}
satisfies the following condition.
\begin{equation}
\begin{array}{ll} \mbox{If} ~ F_a  ~ \mbox{is the intersection of unique $l$ ($0 \leq l \leq n-1$) facets} ~
F_{a_{j_1}}, \ldots, F_{a_{j_l}} ~ \mbox{of} ~ P_a\\
\mbox{then the vectors} ~\vartheta_{a}({F_{a_{j_1}}}), \ldots, \vartheta_{a}({F_{a_{j_l}}})~ 
\mbox{are linearly independent vectors of}~ \mathbb{F}_{2}^{n-1}.
\end{array}
\end{equation}

Hence $ \vartheta_{a} $ is a characteristic function of a small cover over the polytope $P_a$. Since $P_a$ is an
$(n-1)$-simplex, the small cover corresponding to the characteristic pair $(P_a, \vartheta_{a})$
is equivariantly diffeomorphic to the real projective space $ \mathbb{RP}^{n-1} $, see \cite{DJ}.
 Here we consider $\mathbb{RP}^{n-1}$ as the identification space
$\{\overline{B}^{n-1}/ \{x = -x\} : x\in \partial{\overline{B}^{n-1}}\}.$ So at each
point $ (a,0, \ldots, 0)\in B_1^n -\{0\}$ we get an equivariant homeomorphism
\begin{equation}
 (\mathbb{Z}_{2}^{n-1} \times P_a)/\sim_{0} \cong ~\mathbb{RP}^{n-1},
\end{equation}
which sends the fixed point $[t,a]^{\sim_0}$ to the origin of $\overline{B}^{n-1}$.
It is clear from the definition of the equivalence relation $ \sim_{0} $ that at $(0, \ldots, 0) \in B_1^n $, 
$ (\mathbb{Z}_{2}^{n-1} \times P_0)/\sim_{0} $ is a point. Hence $ Y_0 $ is equivariantly
homeomorphic to the open cone
$$ (\mathbb{RP}^{n-1} \times [0,1))/\mathbb{RP}^{n-1} \times \{0\} $$
on real projective space $ \mathbb{RP}^{n-1}$. Consider the following map
$$ S^{n-1} \times [0,1) \to B^n ~\mbox{ define by}~  ((x_1, x_2, \ldots , x_n), r) \to (rx_1, rx_2, \ldots , rx_n). $$
This map induces a homeomorphism $ f : B^n \to (S^{n-1} \times [0,1))/S^{n-1} \times \{0\} $.
The covering map $ S^{n-1} \to \mathbb{RP}^{n-1} $ induces a projection map
$$ \phi_0 :(S^{n-1} \times [0,1))/S^{n-1} \times \{0\} \to (\mathbb{RP}^{n-1} \times [0,1))/\mathbb{RP}^{n-1} \times \{0\}.$$
Observe that this projection map $ \phi_0 $ is nothing but the orbit map $\mathfrak{q}$ of the antipodal
action of $ \ZZ_2 $ on $ B^n $. In other words the following diagram is commutative.
\begin{equation}
\begin{CD}
B^n @>f>> (S^{n-1} \times [0,1))/S^{n-1} \times \{0\} \\
@V\mathfrak{q}VV  @V\phi_0VV \\
B^n /\ZZ_2 @>\hat f>> (\mathbb{RP}^{n-1} \times [0,1))/\mathbb{RP}^{n-1} \times \{0\}
\end{CD}
\end{equation}
Since the map $\phi_0$ is induced from the antipodal action on $S^{n-1}$ the commutativity of the diagram
ensure that the map $ \hat f $ is a homeomorphism. Let $\phi_v $ be the composition of the following maps.
\begin{equation}
\begin{CD}
B^n @>\mathfrak{q}>> B^n /\ZZ_2 @>\hat f>> (\mathbb{RP}^{n-1} \times [0,1))/\mathbb{RP}^{n-1} \times \{0\} @>\cong>> Y_0
@>\hat{f_v}>>X_v.
\end{CD}
\end{equation}
Hence $ (B^n, \ZZ_2, \phi_v) $ is an orbifold chart of $ X_v $ corresponding to the vertex $ v $ of the polytope.

Now we introduce some additional orbifold charts corresponding to each face $ F $ of codimension-$k$ ($ 0 < k < n $)
and the interior of polytope $P$. Let $$ U_F = \bigcap U_v, $$ where the intersection is
over all vertices $ v $ of $ F $. Let $ X_F := {\pi^{-1}}(U_F) $.

Fix a vertex $v$ of $F$. Consider the diffeomorphism $ f_v: B_1^n \to U_v $.
Observe that $U_F$ can be obtained from $U_v$ by deleting unique $n-k$ facets of $U_v$.
Let $F_{l_1}, \ldots, F_{l_{n-k}}$ be the facets of $U_v$ such that
$$ U_F = U_v - \{F_{l_1} \cup \ldots \cup F_{l_{n-k}}\}, $$
where $\{l_1, \ldots, l_{n-k}\} \subset \{1, 2, \ldots, n\}$. Let $B_F^n = f_v^{-1}(U_F) $.
Let $\{x_{l_1}=0\}, \ldots, \{x_{l_{n-k}}=0\}$ be the coordinate hyperplanes in $\RR^n$ such that
$$ f_v(\{x_{l_1} =0\} \cap B_1^n) = F_{l_1},~ \ldots, ~ f_v(\{x_{l_{n-k}} =0\} \cap B_1^n ) = F_{l_{n-k}}.$$
So $ B_F^n = B_1^n - \{ \{x_{l_1}=0\} \cup \ldots \cup \{x_{l_{n-k}} =0\} \}$.
Then $ \hat{f}_v(\pi_0^{-1}(B_F)) = X_F$.

Let $a \in (0,1)$ and $P_a^{\prime} = P_a - \{x_{l_1} =0\}$. Since $(P_a, \vartheta_{a})$ is a
characteristic pair, there exist an equivariant homeomorphism from
$(\ZZ_2^{n-1} \times P_a^{\prime})/ \sim_0$ to $B^{n-1}$ such that
$(\ZZ_2^{n-1} \times F_{a_j})/ \sim_0$ maps to a coordinate hyperplane
$H_j := \{x_{i_j}=0\} \cap B^{n-1}$,
for $j \in \{\{1, 2, \ldots, n\} - l_1\} $. Clearly $H_i \neq H_j$ for $i \neq j$.

Let $P_a^{\prime \prime} = P_a^{\prime} - \{\{x_{l_2} =0\} \cup \ldots \cup \{x_{l_{n-k}}=0\}\}$. Then
$$(\ZZ_2^{n-1} \times P_a^{\prime \prime})/ \sim_0~ \cong B^{n-1} -\{H_{l_2} \cup \ldots \cup H_{l_{n-k}}\}~
\mbox{and} ~ B_F^n \cong (0,1) \times P_a^{\prime \prime}.$$ So
$\pi_0^{-1}(B_F^n) = (\ZZ_2^{n-1} \times B_F^n)/\sim_0$ is homeomorphic to
$$ (0,1) \times \{(\ZZ_2^{n-1} \times P_a^{\prime \prime})/ \sim_0\}
~ \cong ~ (0,1) \times \{B^{n-1} -\{H_{l_2}^{\prime} \cup \ldots \cup H_{l_{n-k}}^{\prime} \}\}.$$
By our assumption
$$ (0,1) \times \{B^{n-1} -\{H_{l_2} \cup \ldots \cup H_{l_{n-k}} \}\} \hookrightarrow
(\mathbb{RP}^{n-1} \times [0,1))/\mathbb{RP}^{n-1} \times \{0\}.$$
So there exist two open subsets $D_F, D_F^{ \prime} $ of $B^n$ such that
$D_F^{\prime } = - D_F$ and the following restrictions are homeomorphism.
\begin{enumerate}
\item $\phi_0 \circ f|_{D_F} : D_F \to (0,1) \times \{B^{n-1} -\{H_{l_2} \cup \ldots \cup H_{l_{n-k}} \}\}$.

\item $\phi_0 \circ f|_{D_F^{\prime}} : D_F^{\prime} \to (0,1) \times \{B^{n-1} -\{H_{l_2} \cup \ldots \cup H_{l_{n-k}} \}\}$.
\end{enumerate}

Hence the restriction  $\phi_v|_{D_F} : D_F \to X_F$ is homeomorphism. Clearly
\begin{equation}
  D_F \cong \{\{B^n \cap \{x_n > 0\}\} - \cup_{j=1, x_{l_j} \neq x_n}^{(n-k-1)} \{{x_{l_j} = 0}\} \}.
\end{equation}
The set $D_F$ is homeomorphic to an open ball in $\RR^n$ if $k=n-1$.
When $k=n-1$, $F$ is an edge of the polytope $P$.
Let $E(P)$ be the set of edges of the polytope $P$  and $e \in E(P)$.
Let $\phi_{e_v} = \phi_v|_{D_e} : D_{e} \to X_e $, where $v \in V(e)$.
Hence $(D_{e}, \{0\}, \phi_{e_v}) $ is an orbifold chart on $X_{e}$ for each $ e \in E(P)$ and $v \in V(e)$.

The set $D_F$ is disjoint union of open sets $\{B_{F(i)}: i= 1, \ldots, 2(n-k-1)\} $ 
in $\RR^n$ whenever $ 0 < k < n-1$. Here all $ B_{F(i)} $ are homeomorphic to an open ball in $\RR^n$. Let
\begin{equation}
\phi_{F_v(i)} = \phi_v|_{B_{F(i)}} : B_{F(i)} \to X_F
\end{equation}
be the restriction of the map $\phi_v$ to the domain $B_{F(i)}$, where $v \in V(F)$.
So for each $(i,v) \in \{1,2, \ldots, 2(n-k-1)\} \times V(F)$, the triple $ (B_{F(i)}, \{0\}, \phi_{F_v(i)}) $
is an orbifold chart on the image of $ \phi_{F_v(i)}$ in $X_F$.

Let $P^0$ be the interior of $P$ and $X_P = \pi^{-1}(P^0)$. Hence the set
\begin{equation}
D_P = \{\{B^n \cap \{x_n > 0\}\} - \cup_{j=1}^{n-1} \{{x_{j} = 0}\}\}
\end{equation}
is homeomorphic to $ X_P$ under the restriction of $\phi_v$ on $D_P$. 
The set $D_P$ is a disjoint union of connected open sets $\{B_{j}: j= 1, \ldots, 2(n-1) \} $ in $\RR^n$
where each $B_j$ is homeomorphic to the open ball $B^n$. Let 
\begin{equation}
 \phi_{P_v(j)}=\phi_v|_{B_{j}} : B_{j} \to X_P
\end{equation}
be the restriction of the map $ \phi_P$ to the domain $B_{j}$.
Hence for $(j,v) \in \{1, 2, \ldots, 2(n-1)\} \times V(P)$, $ (B_{j}, \{0\}, \phi_{P_v(j)}) $ is an orbifold chart
on the image of $ \phi_{P_v(j)}$ in $X_P$. Let
\begin{equation}
\mathfrak{U} = \{(B^n, \ZZ_{2}, \phi_v)\} \cup \{(D_{e}, \{0\}, \phi_{e_v})\} \cup \{(B_{F(i)}, \{0\}, \phi_{F_v(i)})\}
\cup \{(B_j, \{0\}, \phi_{P_v(j)})\}
\end{equation}
where $v \in V(P)$, $e \in E(P)$, $F$ run over the faces of codimension $k$ ($0 < k < n-1$), $i=1,\ldots, 2(n-k-1)$
and  $j=1,\ldots, 2(n-1)$. From the description of orbifold charts corresponding to each faces and interior
of the polytope, it is clear that $\mathfrak{U}$ is an orbifold atlas on $ X $. 
Clearly the inclusions $D_e \hookrightarrow B^n$, $B_{F(i)} \hookrightarrow B^n$ and $B_j \hookrightarrow B^n$
induce the following smooth embeddings respectively:
$$ (D_{e}, \{0\}, \phi_{e_v}) \hookrightarrow (B^n, \ZZ_{2}, \phi_v),
(B_{F(i)}, \{0\}, \phi_{F_v(i)}) \hookrightarrow (B^n, \ZZ_{2}, \phi_v)$$
$$\mbox{and} ~ (B_j, \{0\}, \phi_{P_v(j)}) \hookrightarrow (B^n, \ZZ_{2}, \phi_v).$$
Thus $\mathcal{X}(P, \vartheta) := (X, \mathfrak{U})$ is a smooth orbifold. We denote $\mathcal{X}(P, \vartheta)$
by $\mathcal{X}$ whenever there is no confusion.
\begin{defn}
We call the smooth orbifold $\mathcal{X}(P, \vartheta)$ small orbifold of associated to the
s-characteristic pair $(P, \vartheta)$.
\end{defn}

\begin{remark}
\begin{enumerate}
\item The small orbifold $\mathcal{X}(P, \vartheta)$ is reduced, that is, the group in each chart has an effective action.
Singular set of the orbifold $\mathcal{X}(P, \vartheta)$ is
$$ \Sigma \mathcal{X}(P, \vartheta) = \{[t,v]^{\sim} \in X : v \in V(P)\} .$$
We call an element of $ \Sigma \mathcal{X}(P, \vartheta)$ an orbifold point of $X$.

\item We can not define an s-characteristic function for an arbitrary polytope. For example, the 
$3$-simplex in $\RR^3$ does not admit an s-characteristic function.

\item The small orbifold $X$ is compact and connected.

%\item The small orbifold $X$ is not a quotient of some small cover $M$ over $P$ by some subgroup of $\ZZ_2^n$.
%The fixed point set of the action of $\ZZ_2^n$ on small cover is discrete and the fixed point set of the action
%of $\ZZ_2^{n-1}$ on $X$ is homeomorphic to the $1$-skeleton of $P$.
\end{enumerate}
\end{remark}

\begin{example}\label{eg01}
Let $P^2$ be a simple $2$-polytope in $\RR^2$.
Define
\begin{equation}
\vartheta : \mathcal{F}(P^2) \to \mathbb{Z}_{2} ~ \mbox{by} ~ \vartheta(F) = 1, \forall F \in \mathcal{F}(P^2).
\end{equation}
So $ \vartheta $ is the s-characteristic function of $P^2$. The resulting quotient space $ X(P^2, \vartheta) $
is homeomorphic to the sphere $S^2$. These are the only cases where the identification space is a topological manifold.
\end{example}

\begin{example}\label{eg02}
\begin{figure}[ht]
        \centerline{
           \scalebox{.80}{
            \input{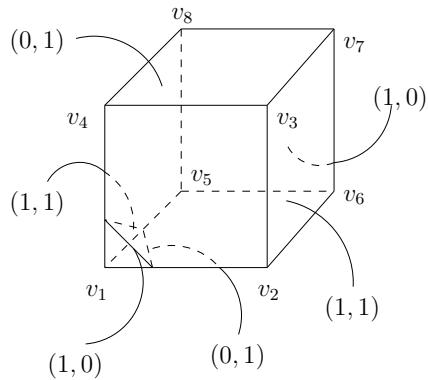}
            }
          }
 \caption{An s-characteristic function of $I^3$.}
 \label{fig-eg2}
 \end{figure}
Let $ I^3 = \{(x,y,z)\in \RR^3 : 0 \leq x,y,z \leq 1\} $ be the standard cube in $ \RR^3 $.
Let $v_1, \ldots, v_8$ be the vertices of $ I^3$, see figure \ref{fig-eg2}.
So the facets of $I^3$ are the following squares
$$F_1 = v_1v_2v_3v_4, F_2 = v_1v_2v_6v_5, F_3 = v_1v_5v_8v_4, F_4 = v_2v_6v_3v_7, F_5 =v_4v_3v_7v_8, F_6 =v_5v_6v_7v_8.$$
Define $ \vartheta : \mathcal{F}(I^3) \to \mathbb{Z}_{2}^2 $ by
$$
\vartheta(F_1) = \vartheta(F_6) = (1,0), ~ \vartheta(F_2) = \vartheta(F_5) = (0, 1),~ \vartheta(F_3) = \vartheta(F_4) =(1, 1).
$$
Hence $ \vartheta $ is an s-characteristic function of $I^3$. Then
$$G_{F_1} = G_{F_6} = \{(0,0),(1,0)\},~ G_{F_2} = G_{F_5} = \{(0,0),(0,1)\},~ G_{F_3} = G_{F_4} = \{(0,0),(1,1)\}.$$

For other proper face $F$ of $ I^3 $, $G_F = \mathbb{Z}_{2}^2 $.
Hence $\mathcal{X}(I^3, \vartheta)$ is a $3$-dimensional small orbifold.
\end{example}

\begin{obser}\label{obssub}
Let $F$ be a codimension-$k$ ($0 < k < n-1$) face of $P$. Then 
%$P^{\prime}$ is the intersection of unique $k$ facets $F_{i_1}, \ldots, F_{i_k}$ of $P$ and
$F$ is a simple polytope of dimension $n-k$.
Let $\mathcal{F}(F) = \{F_{j_1}^{\prime}, \ldots, F_{j_l}^{\prime}\}$
be the set of facets of $F$. So there exist unique facets $F_{j_1}, \ldots, F_{j_l} $ of $P$ such that
$$F_{j_1} \cap F =  F_{j_1}^{\prime},~ \ldots, ~ F_{j_l} \cap F = F_{j_l}^{\prime}.$$
Fix an isomorphism $\mathfrak{b} $ from the quotient field $ \mathbb{F}_2^{n-1}/ G_F $ to $ \mathbb{F}_2^{n-1-k}$.
Define a function 
$$\vartheta^{\prime} : \mathcal{F}(F) \to \mathbb{Z}_2^{n-1-k} ~ \mbox{by}
~ \vartheta^{\prime}(F_{j_i}^{\prime}) = \mathfrak{b}(\vartheta(F_{j_i}) + G_F).$$
Observe that the function $\vartheta^{\prime}$ is an s-characteristic function of $F$.
Let $\sim^{\prime}$ be the restriction of $\sim$ on $\ZZ_2^{n-1-k} \times F$.
So $\mathcal{X}(F, \vartheta^{\prime})$ is an $(n-k)$-dimensional smooth small orbifold associated to the
s-characteristic pair $(F, \vartheta^{\prime})$. The orbifold $\mathcal{X}(F, \vartheta^{\prime})$ is a suborbifold
of $\mathcal{X}(P, \vartheta)$ and the underlying space $X(F, \vartheta^{\prime}) = \pi^{-1}(F)$. We have shown that for each edge $e$ of $P$, the set $X_e$ is homeomorphic to the
open ball $B^n$. Let $e^{\prime}$ be an edge of $F$ and $U_{e^{\prime}}^{\prime} = U_{e^{\prime}} \cap F$. Hence
$$W_{e^{\prime}} = (\ZZ_2^{n-1-k} \times U_{e^{\prime}}^{\prime})/ \sim^{\prime} 
= (\ZZ_2^{n-1} \times U_{e^{\prime}}^{\prime})/ \sim $$
is homeomorphic to the open ball $B^{n-k}$.
\end{obser}

\section{Orbifold fundamental group}\label{orbf}
Orbifold cover and orbifold fundamental group was introduced by Thurston in \cite{Th}.
In this section we compute the universal orbifold cover and orbifold fundamental group of an $n$-dimensional
($n \geq 3$) small orbifold $ \mathcal{X} $ over $P\subset \RR^n$.
\begin{defn}
A covering orbifold or orbifold cover of an $n$-dimensional orbifold $\mathcal{Z}$ is a
smooth map of orbifolds  $\mathfrak{g}: \mathcal{Y} \to \mathcal{Z}$
whose associated continuous  map $g: Y \to Z$ between underlying spaces satisfies
the following condition.

 Each point $z \in Z$ has a neighborhood $U \cong V/\Gamma$ with $V$ homeomorphic to a connected open set in $\RR^n$,
for which each component $W_i$ of $g^{-1}(U)$ is homeomorphic to $V/\Gamma_i$ for some subgroup $\Gamma_i \subset \Gamma$
such that the natural map $g_i: V/\Gamma_i \to V/\Gamma$ corresponds to the restriction of $g$ on $W_i$.
\end{defn}
\begin{defn}
Given an orbifold cover $\mathfrak{g}: \mathcal{Y} \to \mathcal{Z}$ a diffeomorphism
 $\mathfrak{h}: \mathcal{Y} \to \mathcal{Y}$ is called a deck transformation if
$\mathfrak{g} \circ \mathfrak{h} = \mathfrak{g}$.
\end{defn}
\begin{defn}
An orbifold cover $\mathfrak{g}: \mathcal{Y} \to \mathcal{Z}$ is called a universal
orbifold cover of $\mathcal{Z}$ if given any orbifold cover $\mathfrak{g}_1:\mathcal{W} \to \mathcal{Z}$,
there exists an orbifold cover $\mathfrak{g}_2: \mathcal{Y} \to \mathcal{W}$
such that $\mathfrak{g} = \mathfrak{g}_1 \circ \mathfrak{g}_2$.
\end{defn}
Every orbifold has a universal orbifold cover which is unique up to diffeomorphism, see \cite{Th}.
The corresponding group of deck transformations is called the orbifold fundamental group of
$\mathcal{Z}$ and denoted $\pi_1^{\rm orb} (\mathcal{Z})$.

The set of smooth points $$ M := X - \Sigma \mathcal{X}$$ of small
orbifold $ \mathcal{X} $ is an $n$-dimensional manifold. For each $v \in V(P)$ we have
$$X_v - [0,v]^{\sim} \cong \mathbb{RP}^{n-1} \times I^0.$$
The sphere $S^{n-1}$ is the double sheeted universal cover of $\mathbb{RP}^{n-1}$.
So the universal cover of $ X_v -[0, v]^{\sim}$ is $S^{n-1} \times I^0 \cong B^n - 0$.
Let $e$ be an edge containing the vertex $v$ of $P$. Define $\bar{e} : = e \cap U_v$.

Identifying the faces containing the edge $\bar{e}$ of $U_v$ according to the equivalence relation $\sim$
we get the quotient space $ X_{\bar{e}}(U_v, \vartheta) $ homeomorphic to
$$B^n_{e} := \{(x_1, x_2, \ldots, x_n) \in B^n : x_n \geq 0\}.$$
The set $ X_v $ is obtained from $ X_{\bar{e}}(U_v, \vartheta) $ by identifying the antipodal
points of the boundary of $ X_{\bar{e}}(U_v, \vartheta) $ around the fixed point $ [a,v]^{\sim} $.
Identifying two copies of $ X_{\bar{e}}(U_v, \vartheta) $ along their boundary via the antipodal map
on the boundary we get a space homeomorphic to $ B^n $.

Doing these identification associated to the orbifold points we obtain that the universal
cover of $ M $ is homeomorphic to $ \RR^n - N $ for some infinite subset $ N $ of $ \ZZ^n $ where $N$
depends on the polytope $P$ in $ \RR^n $. Let
\begin{equation}
\zeta: \RR^n - N \to M
\end{equation}
be the universal covering map.

The chart maps $\phi_v : B^n \to X_v$ are uniformly continuous and $P$ is an
$n$-polytope in $\RR^n$. So for each $x \in N$ there exists a neighborhood $V_x \subset \RR^n$
of $x$ such that the restriction of the universal covering map $\zeta$ on $V_x -x$ is uniformly
continuous. Hence the map $\zeta$ has a unique extension, say $ \hat{\zeta} $, on their metric
completion. The metric completion of $ \RR^n - N $ and $ M $ are $ \RR^n $ and
$ X$ respectively. The map $ \hat{\zeta} $ sends $ N $ onto $ V(P)$.

We show the map $ \hat{\zeta} $ is an orbifold covering.
Let $ \varrho : \mathcal{Z} \to \mathcal{X} $ be an orbifold cover.
Then the restriction $ \varrho : Z - \Sigma \mathcal{Z} \to M $ is an honest cover.
Hence there exist a covering map $ \zeta_{\varrho} : \RR^n - N \to Z - \Sigma \mathcal{Z} $
so that the following diagram is commutative.
\begin{equation}\label{unicov}
\begin{CD}
\RR^n - N @>\zeta_{\varrho}>> Z - \Sigma \mathcal{Z} \\
@V\zeta VV  @V \varrho VV \\
M @>id>> M
\end{CD}
\end{equation}

Since the map $\zeta$ is locally uniformly continuous and the maps $\zeta_{\varrho}, \varrho$ are continuous,
all the maps in the diagram \ref{unicov} can be extended to their metric completion.
That is we get a commutative diagram of orbifold coverings.
\begin{equation}
\begin{CD}
\RR^n @>\hat{\zeta}_{\varrho}>> \mathcal{Z}\\
@V\hat{\zeta}VV  @V\hat{\varrho}VV\\
\mathcal{X} @>id>> \mathcal{X}
\end{CD}
\end{equation}
Hence  $ \hat{\zeta}: \RR^n \to \mathcal{X}$ is an orbifold universal cover of $ \mathcal{X}$. Clearly the map
$ \hat{\zeta} $ is a smooth map.
\begin{theorem}\label{uniorb}
The universal orbifold cover of an $n$-dimensional small orbifold is diffeomorphic to $\RR^n$.
\end{theorem}
\begin{figure}[ht]
 \centerline{
  \scalebox{0.80}{
  \input{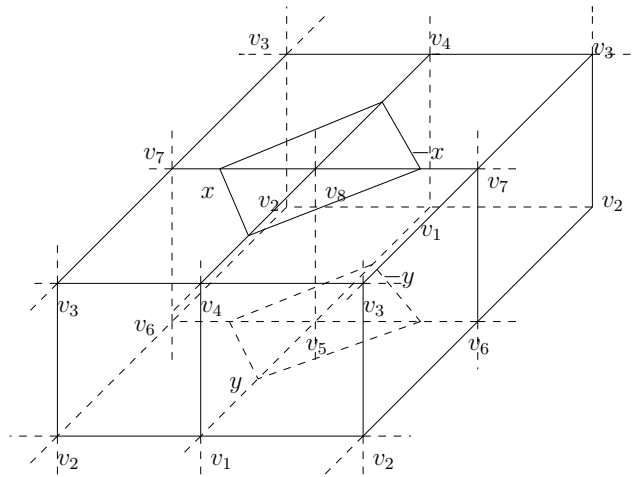}
   }
   }
 \caption {Identification of faces containing the edge $v_5v_8$ of $I^3$.}
  \label{fig-eg}
 \end{figure}
\begin{example}
Recall the small orbifold $X(I^3, \vartheta)$ of example \ref{eg02}. The set of smooth points
$$M(I^3, \vartheta) := X(I^3, \vartheta) - \Sigma \mathcal{X}(I^3, \vartheta)$$
is a $3$-dimensional manifold. The universal cover of $ M(I^3, \vartheta) $ is
homeomorphic to $\RR^3-\ZZ^3$. To show this we need to observe how the faces of $ \mathbb{Z}_{2}^2 \times I^3 $ are
identified by the equivalence relation $ \sim $ (see equation \ref{idenso}) on $ \mathbb{Z}_{2}^2 \times I^3 $.
For each $v \in V(I^3)$
$$ X_v(I^3, \vartheta) -[a,v]^{\sim} \cong \mathbb{RP}^2 \times I^0.$$
The sphere $S^2$ is the double sheeted universal cover of $\mathbb{RP}^2$. So the universal cover of
$ X_v(I^3, \vartheta) -[a,v]^{\sim}$ is $S^2 \times I^0 \cong B^3 - 0$. Hence the identification of faces
around each vertex of $I^3$ tells us that the universal cover of $M(I^3, \vartheta) $ is $ \RR^3 - \ZZ^3 $.
We illustrate the identification of faces by the figure \ref{fig-eg}, where $x \sim -x$ on the upper face
and $y \sim -y$ on the lower face in that figure.
\end{example}

Now we use the following observation from \cite{ALR} to compute the orbifold fundamental group
of $\mathcal{X}$.
\begin{obser}\label{orbgp2}
Suppose that $\mathfrak{u} : \mathcal{\widehat{Y}} \to \mathcal{Y}$  is an orbifold universal cover.
Then the restriction $\mathfrak{u} : \mathcal{\widehat{Y}} - \Sigma \mathcal{\widehat{Y}} \to
\mathcal{Y} - \Sigma \mathcal{Y}$ is an honest cover with $G = \pi_{1}^{\rm orb}(\mathcal{Y})$
as the orbifold covering group, where $ \Sigma \mathcal{Y} $ is the singular subset of $ \mathcal{Y} $.
Therefore $ \mathcal{Y} =  \mathcal{\widehat{Y}}/ G $.
\end{obser}

Let $\{\beta_1, \beta_2, \ldots, \beta_m\}$ be the standard basis of $\ZZ_2^m$.
Define a map $ \beta : \mathcal{F}(P) \to \ZZ_2^m$ by $\beta(F_j) = \beta_j.$ For each face
$F = F_{j_1} \cap F_{j_2} \cap \ldots \cap F_{j_l}$, let $H_F$ be the subgroup
of $\ZZ_2^m$ generated by $\beta_{j_1}, \beta_{j_2}, \ldots, \beta_{j_l}$. Define
an equivalence relation $\sim_{\beta}$ on $\ZZ_2^m \times P$ by
$$(s,p) \sim_{\beta} (t,q)~ \mbox{if and only if}~ p=q~ \mbox{and}~ t-s \in H_F$$
where $F \subset P$ is the unique face whose relative interior contains $p$. So the
quotient space $ N(P, \beta) = (\ZZ_2^m \times P)/ \sim_{\beta} $ is an $n$-dimensional
smooth manifold. $N(P, \beta)$ is a $\ZZ_2^m$-space with the orbit map
$$\pi_u : N(P, \beta) \to P ~\mbox{defined by}~ \pi_u ([s,p]^{\sim_{\beta}}) = p.$$

We show $P$ has a smooth orbifold structure. Recall the open subset $U_v$ of $P$ associated to each vertex $v \in V(P)$.
Note that open sets $\{U_v : v\in V(P)\}$ cover $P$.
Let $d$ be the Euclidean distance in $\RR^n$. Let $F_{i_1}, F_{i_2}, \ldots, F_{i_n}$ be the facets
of $P$ such that $ v$ is the intersection of $ F_{i_1}, F_{i_2}, \ldots, F_{i_n}$. For each $p \in U_v$, let
$$x_j(p) = d(p, F_{i_j}), ~ \mbox{for all} ~ j = 1, 2, \ldots, n.$$
Let $B_v^n = \{(x_1(p), \ldots, x_n(p)) \in \RR^n_{\geq 0} : p \in U_v \}$. So the map
$$f : U_v \to B_v^n~ \mbox{defined by}~ p \to (x_1(p), \ldots, x_n(p))$$
gives a diffeomorphism from $U_v$ to $B_v^n$. Consider the standard action of $\ZZ_2^n$ on $\RR^n$ with
the orbit map $$\xi : \RR^n \to \RR^n_{\geq 0}.$$ Then $\xi^{-1}(B_v^n)$ is diffeomorphic to $B^n$.
Hence $(\xi^{-1}(B_v^n), \xi \circ f^{-1}, \ZZ_2^n)$ is a smooth orbifold chart on $U_v$.
To show the compatibility of these charts as $v$ varies over $V(P)$, we can introduce some additional
smooth orbifold charts to make this collection an smooth orbifold atlas as in section \ref{def}.
From the definition of $\sim$ it is clear that $\pi : X(P, \vartheta) \to P$ is a smooth orbifold covering.
\begin{defn}
Let $L$ be the simplicial complex dual to $P$. The right-angled Coxeter group $\Gamma$ associated to $P$ is the
group with one generator for each element of $V(L)$ and relations between generators are the following;
$a^2=1$ for all $a \in V(L)$, $(ab)^2 = 1$ if $\{a, b\} \in E(L)$.
\end{defn}
For each $p \in P$, let $F(p) \subset P$ be the unique face containing $p$ in its relative interior.
Let $F(p) = F_{j_1} \cap \ldots \cap F_{j_l}$. Let $a_{j_1}, \ldots, a_{j_l}$ be the vertices of $L$
dual to $ F_{j_1}, \ldots, F_{j_l}$ respectively. Let $\Gamma_{F(p)}$ be the subgroup
generated by $a_{j_1}, \ldots, a_{j_l}$ of $\Gamma$. Define an equivalence relation $\sim$ on
$\Gamma \times P$ by $$(g, p) \sim (h, q)~\mbox{if}~ p=q ~\mbox{and}~ h^{-1}g \in \Gamma_{F(p)}.$$
Let $Y= (\Gamma \times P)/\sim$ be the quotient space. We follow this construction from \cite{DJ}.
So $Y$ is a $\Gamma$-space with the orbit map
\begin{equation}
\xi_{\Gamma} : Y \to P ~ \mbox{defined by} ~ \xi_{\Gamma}([g,p]^{\sim}) = p.
\end{equation}
Then $Y$ is an $n$-dimensional manifold and $\xi_{\Gamma}$ is an orbifold
covering. Since each facet is connected, whenever two generators of $\Gamma$ commute
the intersection of corresponding facets of $P$ is nonempty. From Theorems 10.1 and 13.5 of \cite{D},
we get that $Y$ is simply connected. Hence $\xi_{\Gamma}$ is a universal orbifold covering and the
orbifold fundamental group of $P$ is $\Gamma$.

Let $H$ be the kernel of abelianization map $\Gamma \to \Gamma^{ab}$. The group $H$ acts on $Y$ freely and
properly discontinuously. So the orbit space $Y/H$ is a manifold. The space $Y/H$ is called the universal
abelian cover of $P$. Note that $N(P, \beta) = Y/H$. Let
\begin{equation}
 \xi_{\beta} : Y \to N(P, \beta)
\end{equation}
be the corresponding orbit map.
 
Define a function $\bar{\vartheta} : \ZZ_2^m \to \ZZ_2^{n-1}$ by
$\bar{\vartheta}(\beta_j) = \vartheta (F_j) = \vartheta_j$ on the basis of $\ZZ_2^m$.
So $\bar{\vartheta}$ is a linear surjection. $\bar{\vartheta}$ induces a surjection
\begin{equation}
\widetilde{\vartheta} : N(P, \beta) \to X(P, \vartheta)~\mbox{defined by}~
\widetilde{\vartheta}([s,p]^{\sim_{\beta}})= [s,p]^{\sim}.
\end{equation}
That is the following diagram commutes.
\begin{equation}
\begin{CD}
N(P, \beta) @>\widetilde{\vartheta}>> X(P, \vartheta) \\
@V\hat \pi_u VV  @V\hat \pi VV \\
 P @>id >> P
\end{CD}
\end{equation}
From this commutative diagram we get $\widetilde{\vartheta}$ is a smooth orbifold covering of $X(P, \vartheta)$.
Hence the composition map $$\widetilde{\vartheta} \circ \xi_{\beta} : Y \to X(P, \vartheta)$$ is a smooth
universal orbifold covering. From \cite{Th} and Theorem \ref{uniorb} we obtain the following necessary
condition for existence of an s-characteristic function.
\begin{theorem}\label{isos}
Let $\vartheta : \mathcal{F}(P) \to \mathbb{Z}_{2}^{n-1}$ be an s-characteristic function of the
$n$-polytope $P$ ($n > 2$). Then the space $Y$ is diffeomorphic to $\RR^n$. 
\end{theorem}

Note that when $P$ is an $n$-simplex, $Y$ is homeomorphic to the $n$-dimensional sphere $S^n$.
So by this theorem there does not exist an s-characteristic function of $n$-simplex.

Let $ \xi_{\vartheta}$ be the following composition map
\begin{equation}
\begin{CD}
 \Gamma \to \Gamma^{ab} @>\bar{\vartheta} >> \ZZ_2^{n-1}.
\end{CD}
\end{equation}
Clearly $ker(\xi_{\vartheta})$, kernel of $ \xi_{\vartheta}$, acts on $Y$ with the orbit map
$\widetilde{\vartheta} \circ \xi_{\beta}$. Now using the observation \ref{orbgp2}, we get the following corollary.
\begin{corollary}
The orbifold fundamental group of $X(P, \vartheta)$ is $ker(\xi_{\vartheta})$ which is
a normal subgroup of the right-angled Coxeter group associated to the polytope $P$.
\end{corollary}

\section{Homology and Euler characteristic}\label{hom}
\subsection{Face vector of polytope}
The face vector or $f$-$vector$ is an important concept in the combinatorics of polytopes.
Let $L$ be a simplicial $n$-polytope and $ f_j $ is the number of $j$-dimensional faces of $L$.
The integer vector $f(L) = (f_0, \ldots, f_{n-1})$ is called the $f$-$vector$ of the simplicial polytope $L$.
Let $h_i$ be the coefficients of $ t^{n-i}$ in the polynomial
\begin{equation}
 (t-1)^n + \Sigma_{0}^{n-1} f_i (t-1)^{n-1-i}.
\end{equation}
The vector $h(L) = (h_0, \ldots, h_n)$
is called $h$-$vector$ of $ L$. Obviously $h_0 = 1$, and $\Sigma_0^n h_i = f_{n-1}$.
The $f$-$vector$ and $h$-$vector$ of a simple $n$-polytope $P$ is the
$f$-$vector$ and $h$-$vector$ of its dual simplicial polytope respectively,
 that is $ f(P) = f(P^{\ast})$ and  $ h(P) = h(P^{\ast})$.

Hence for a simple $n$-polytope $P$,
\begin{equation}
 f(P) = (f_0, \ldots, f_{n-1}),
\end{equation}
where $f_j$ is the number of codimension-$(j+1)$ faces of $P$.
Then $h_n = 1$ and $ \Sigma_{_{1}}^{^{n}} h_i $ is the number of vertices of $P$.
The face vectors are a combinatorial invariant of polytopes, that is it depends only on the face poset
of the polytope.

\subsection{$CW$-complex structure }\label{cw}
To calculate the singular homology groups of small orbifold $X$ we construct a $CW$-structure on these
orbifolds and describe how the cells are attached.
Realize $P$ as a convex polytope in $\RR^n$ and choose a linear functional
\begin{equation}\label{indfun}
 \phi: \RR^n \to \RR
\end{equation}
which distinguishes the vertices of $P$, as in proof of Theorem $3.1$ in \cite{DJ}.
The vertices are linearly ordered according to ascending value of $\phi$.
We make the $1$-skeleton of $P$ into a directed graph by orienting each
edge such that $\phi$ increases along edges.
For each vertex of $ P $ define its index $ ind_{P}(v) $,
as the number of incident edges that point towards $ v $.
\begin{defn}
A subset $Q \subseteq P$ of dimension $k$ is called a proper subcomplex of $P$ if $Q$ is
connected and $Q$ is the union of some $k$-dimensional faces of $P$.
\end{defn}

In particular each face of $P$ is a proper subcomplex of $P$.
The $1$-skeleton of $Q$ is a subcomplex of the $1$-skeleton of $P$.
The restriction of $ \phi $ on the $1$-skeleton of $ Q $ makes it a directed graph.
Define index $ ind_{Q}(v) $ of each vertex $v$ of $Q$ as the number of incident
edges in $Q$ that point towards $ v $. Let $V(Q)$ and $\mathfrak{F}(Q)$ denote the
set of vertices and the set of faces of $Q$ respectively.
We construct a $CW$-structure on $X$ in the following way. Let 
$$ I_{P} = \{(u,e_u) \in V(P)\times E(P) : ind_P(u) = n~ \mbox{and}~ e_u ~\mbox{is the edge joining the vertices}$$
$$ u, x_u ~\mbox{such that} ~ \phi(u) > \phi(x_u) > \phi(a)~\mbox{for all vertex} ~ a \in V(P)-\{u,x_u\} \}.$$

Let $ U_{e_u} = U_{u} \cap U_{x_u} $ and $ Q^n = P$.
Then $ W_{e_u} = (\mathbb{Z}_{2}^{n-1} \times U_{e_u})/ \sim $ is homeomorphic to the $n$-dimensional
open ball $ B^n \subset \RR^n $. Let
\begin{equation}
  Q^{n-1} = P - U_{e_u}.
\end{equation}
$Q^{n-1}$ is the union of facets not containing the edge $ e_u $ of $P$. So $ Q^{n-1} $ is an
$(n-1)$-dimensional proper subcomplex of $P$ and $ V(P) = V(Q^{n-1})$.
Let $v \in V(Q^{n-1})$ with $ ind_{Q^{n-1}}(v) = n-1 $. Let $F_v^{n-1} \in \mathfrak{F}(Q^{n-1})$
be the smallest face which contains the inward pointing edges incident to $v$ in $Q^{n-1}$. If
$v_1, v_2$ are two vertices of $Q^{n-1}$ with $ ind_{Q^{n-1}}(v_1) = n-1 = ind_{Q^{n-1}}(v_2)$
then $F_{v_1}^{n-1} \neq F_{v_2}^{n-1}$. Let 
$$ I_{Q^{n-1}} = \{(v,e_v) \in V(P) \times E(P) : ind_{Q^{n-1}}(v) = n-1 ~  \mbox{and}~e_v ~
 \mbox{is the edge joining the}$$
$$\mbox{vertices}~ v, y_v \in V(F_v^{n-1})~\mbox{such that}~ \phi(v) > \phi(y_v) > \phi(b)~ \forall ~
b \in V(F_v^{n-1})-\{v, y_v\}\}.$$
Let $$ U_{e_{v}} = U_{v} \cap U_{y_v} \cap F_v^{n-1} ~ \mbox{for each}~ (v, e_v) \in I_{Q^{n-1}}.$$
From observation \ref{obssub}, $ W_{e_v} = (\mathbb{Z}_{2}^{n-1} \times U_{e_v})/ \sim $ is homeomorphic
to the $(n-1)$-dimensional open ball $ B^{n-1} \subset \RR^{n-1} $. Let
\begin{equation}
Q^{n-2} = P - \{\{ \bigcup_{(u, e_u) \in I_{Q^n}} U_{e_{u}}\} \cup \{\bigcup_{(v, e_v) \in I_{Q^{n-1}}} U_{e_{v}}\}\}.
\end{equation}
So $ Q^{n-2} $ is an $(n-2)$-dimensional proper subcomplex of $P$ and $ V(P) = V(Q^{n-2})$.
Let $w \in V(Q^{n-2})$ with $ ind_{Q^{n-2}}(w) = n-1 $. Let $F_w^{n-2} \in \mathfrak{F}(Q^{n-2})$
be the smallest face which contains the inward pointing edges incident to $w$ in $Q^{n-2}$. If
$w_1, w_2$ are two vertices of $Q^{n-2}$ with $ ind_{Q^{n-2}}(w_1) = n-1 = ind_{Q^{n-2}}(w_2)$
then $F_{w_1}^{n-2} \neq F_{w_2}^{n-2}$. Let
$$ I_{Q^{n-2}} = \{(w,e_w) \in V(P) \times E(P) : ind_{Q^{n-2}}(w) = n-2 ~ \mbox{and}~e_w~\mbox{is the edge joining the}$$
$$\mbox{vertices} ~w, z_w \in V(F_w^{n-2}) ~ \mbox{such that}~\phi(w) > \phi(z_w) > \phi(c)~
\forall~ c \in V(F_w^{n-2})- \{w,z_w\} \}.$$
Let $$ U_{e_{w}} = U_{w} \cap U_{z_w} \cap F_w^{n-2} ~ \mbox{for each}~ (w, e_w) \in I_{Q^{n-2}}.$$
From observation \ref{obssub}, $ W_{e_w} = (\mathbb{Z}_{2}^{n-1} \times U_{e_w})/ \sim $ is homeomorphic
to the $(n-2)$-dimensional open ball $ B^{n-2} \subset \RR^{n-2} $.

Continuing this process we observe that $ Q^1 (\cong (\mathbb{Z}_{2}^{n-1} \times Q^1)/ \sim ) $
is a maximal tree of the $1$-skeleton of $P$ and $ Q^0 = V(P)$.
Hence relative interior of each edge of $(\mathbb{Z}_{2}^{n-1} \times Q^1)/ \sim $ is homeomorphic
to the $1$-dimensional ball in $\RR$.
So corresponding to each edge of polytope $P$, we construct a cell of dimension $\geq 1 $ of $X$.

The integer $h_{n-i}$ is the number of vertices $v \in V(P)$ of $ind_P(v) =i$. The Dehn-Sommervile relation is
$$h_i=h_{n-i}~ \forall~ i=0,1,\ldots, n,$$
see Theorem $1.20$ of \cite{BP}. Hence the number of $k$-dimensional cells in $X$ is
\begin{equation}
|I_{Q^{k}}| = \Sigma_{k}^{n} h_i.
\end{equation}

We describe the attaching map for a $k$-dimensional cell. Here $k$-dimensional cells are
$$\{ W_{e_v} : (v,e_v) \in I_{Q^k} \}.$$
Let $(v, e_v) \in I_{Q^k}$. Let $ F_v^k \in \mathfrak{F}(Q^k) $ be the smallest face containing the inward
pointing edges to $v$ in $Q^k$.
Define an equivalence relation $\sim_{e_v}$ on $\ZZ_2^{n-1} \times F_v^k$ by
\begin{equation}\label{equiev}
(t,p) \sim_{e_v} (s,q) ~\mbox{if}~ p=q \in F^{\prime} ~\mbox{and} ~ s-t \in G_{F^{\prime}} 
\end{equation}
where $ F^{\prime} \in \mathfrak{F}(F_v^k)$ is a face containing the edge $e_v$.
The quotient space $ (\mathbb{Z}_{2}^{n-1} \times F_v^k)/\sim_{e_v} $ is homeomorphic to the closure
of open ball $B^k \subset \RR^k$. The attaching map $\phi_{F_v^k}$ is the natural quotient map
\begin{equation}\label{atcmap}
\phi_{F_v^k} : S^{k-1} \cong (\mathbb{Z}_{2}^{n-1}\times (F_v^k- U_{e_v})/\sim_{e_v} ~ \to
 (\mathbb{Z}_{2}^{n-1} \times (F_v^k- U_{e_v})/\sim.
\end{equation}
Let $ X_k = {\displaystyle \bigcup_{i=1}^{k} ~ \bigcup_{(v, e_v) \in I_{Q^i}}{\bar{W}_{e_v}}}$.
Then $X_k$ is the $k$-th skeleton of $X$ and
$$ X = \displaystyle \bigcup_{k=1}^{n} X_k.$$
So we get a $CW$-complex structure on $ X $ with $ \Sigma_{k}^{n} h_i $ cells in dimension $k$, $ 0\leq k \leq n $.
Since singular homology and cellular homology are isomorphic, we compute cellular homology of $X$.
To calculate cellular homology we compute the boundary map of the cellular chain complex for $X$.
To compute the boundary map we need to compute the degree of the following composition map $\beta_{we_v}$
\begin{equation}\label{bddmap}
\begin{CD}
S^{k-1} @>\phi_{F_v^k}>> (\ZZ_{2}^{n-1}\times (F_v^k- U_{e_v})/\sim @>q>> \frac{X_{k-1}}{X_{k-2}}
= \displaystyle \bigvee_{(w, e_w) \in I_{Q^{k-1}}} S^{k-1}_w @>q_{w}>> S^{k-1}_w
\end{CD}
\end{equation}
where $F_v^k$ is a face of $Q^k$ of dimension $k$ ($k \geq 2$), $S_w^{k-1} \cong S^{k-1}$ and $ q, q_w $ are
the quotient maps. Clearly the above composition map $\beta_{we_v}$ is either surjection or constant up to homotopy.
When the map is constant the degree of the composition map $\beta_{we_v}$ is zero.
We calculate the degree of the composition when it is a surjection.

Let $(w, e_w) \in I_{Q^{k-1}}$ be such that $\beta_{we_v}$ is a surjection.
Let $ z_w$ be the vertex of the edge $e_w$ other than $w$.
Let $ F_w^{k-1} \in \mathfrak{F}(Q^{k-1})$ be the smallest face which contains the inward
pointing edges to $w$ in $Q^{k-1}$. Let $$U_{e_w} = U_w \cap U_{z_w} \cap F_w^{k-1}.$$
So $ U_{e_w}$ is an open subset of $F_w^{k-1} $ and $U_{e_w}$ contains the relative interior
of $e_w$. The face $F_w^{k-1} \subset F_v^k -U_{e_v}$ is a facet of $F_v^k$. Note that
$$ W_{e_w} = (\ZZ_2^{n-1} \times U_{e_w})/ \sim ~= S^{k-1}_w - \{X_{k-2}/X_{k-2}\}.$$

The quotient group $G_{F_w^{k-1}}/G_{F_v^k}$ is isomorphic to $\ZZ_2$.
Hence from equations \ref{equiev}, \ref{atcmap} and \ref{bddmap} we get that
$ (\beta_{we_v})^{-1}(W_{e_w}) $ has two components $ Y^1 $ and $ Y^2 $ in $S^{k-1}$.
The restrictions $ (\beta_{we_v})_{|Y^1} $ and $ (\beta_{we_v})_{|Y^1} $, on $ Y^1 $ and $ Y^2 $
respectively, give homeomorphism to $ W_{e_w} $. Let $ y_v$ be the vertex of the edge
$e_v$ other than $v$. Observe that
$$(\mathbb{Z}_{2}^{n-1}\times (F_v^k- \{U_{e_v} \cup \{v, y_v\}\})/\sim_{e_v} \cong I^0 \times S^{k-2}.$$

Hence from the definition of equivalence relation $\sim$, 
it is clear that $ Y^2 $ is the image of $ Y^1 $ under the map (possibly up to homotopy)
\begin{equation}
(id \times \mathfrak{a}) : I^0 \times S^{k-2} \to I^0 \times S^{k-2} ~\mbox{defined by}~
(id \times \mathfrak{a}) (r, x) = (r, -x),
\end{equation}
where $I^0$ is the open interval $(0,1) \subset \RR$. The degree of $(id \times \mathfrak{a})$ is $ (-1)^{k-1} $.
Hence the degree of the composition map $\beta_{we_v}$ is
\begin{equation}
d_{vw} = deg(\beta_{we_v}) = \left\{ \begin{array}{ll} 2 & \mbox{if} ~ k\geq 2 ~ \mbox{is odd and}
~ \beta_{a_v} ~ \mbox{is a surjection}\\
 0 & \mbox{if} ~ k\geq 2 ~ \mbox{is even and} ~ \beta_{a_v} ~ \mbox{is a surjection}\\
 0 & \mbox{if} ~ \beta_{a_v} ~ \mbox{is constant}
\end{array} \right.
\end{equation}

Hence the cellular chain complex of the constructed $CW$-complex of small orbifold $X$ is 
\begin{equation}\label{ccc}
\begin{CD}
0 \to \ZZ @>d_n>> \oplus_{|I_{Q^{n-1}}|} ~\ZZ \to \cdots @>d_3>> \oplus_{|I_{Q^2}|}~\ZZ
@>d_2>> \oplus_{|I_{Q^1}|}~\ZZ @>d_1>> \oplus_{|I_{Q^0}|}~\ZZ \to 0\\
\end{CD}
\end{equation}
where $ d_k $ is the boundary map of the cellular chain complex. 
If $ k\geq 2 $ the formula of $d_k$ is
\begin{equation}
 d_k(W_{e_v}) =\displaystyle \sum_{(w, e_w)\in I_{Q^{k-1}}} d_{vw} W_{e_w},
\end{equation}
 where $(v, e_v) \in I_{Q^k}$ and $d_{vw}$ is the degree of the composition map $ \beta_{we_v} $.
Hence the map $ d_{k} $ is represented by the following matrix
%of order $\displaystyle\sum_{_{k}}^{^{n}} h_i \times \displaystyle\sum_{_{k-1}}^{^{n}} h_i $ 
with entries
\begin{equation}
 \{d_{vw} : (v, e_v)\in I_{Q^k}, (w,e_w)\in I_{Q^{k-1}} \}.
\end{equation}
 
So the map $ d_{k} $ is the zero matrix for all even $ k $.
When $ k \geq 2 $ is odd, the map $ d_{k} $ is injective and the image of the map $ d_{k} $
is the submodule generated by
\begin{equation}
\{\displaystyle \sum_{(w, e_w)\in I_{Q^{k-1}}} d_{vw} W_{e_w} : (v, e_v)\in I_{Q^k}\}.
\end{equation}
Hence the quotient module $(\oplus_{|I_{Q^{k-1}}|}~ \ZZ)/ Im d_k$ is isomorphic to 
$(\oplus_{h_k}~ \ZZ) \oplus (\oplus_{\Sigma_{k+1}^{n} h_i}~ \ZZ_2)$.

Since the $1$- skeleton $ X_1 $ is a tree with $ \Sigma_0^{n} h_i $ vertices and
$\Sigma_{1}^{n} h_i $ edges. The boundary map $ d_1 $ is an injection. The image of $ d_1 $ is
$ \Sigma_{1}^{n} h_i $ dimensional direct summand of $ \oplus_{|I_{Q^0}|} ~\ZZ $ over $ \ZZ $.
Hence $ (\oplus_{|I_{Q^0}|} ~\ZZ)/ d_1( \oplus_{|I_{Q^1}|} ~\ZZ ) $ is isomorphic to $ \ZZ $.
From the previous calculation we have proved the following theorem.
\begin{theorem}
The singular homology groups of the small orbifold $ X $ with coefficients in $\ZZ$ is
$$H_k(X, \ZZ) = \left\{ \begin{array}{ll} \ZZ & \mbox{if} ~ k = 0 ~ \mbox{and if} ~ k = n ~ \mbox{even}\\
(\oplus_{h_k} ~\ZZ) \oplus (\oplus_{\Sigma_{k+1}^{n} h_i} ~\ZZ_2) & \mbox{if} ~ k ~ \mbox{is even},~ 0 <  k  <  n\\
 0 & \mbox{otherwise}
\end{array} \right.
$$
\end{theorem}
\begin{remark}
If $P$ is an even dimensional simple polytope then the small orbifold over $P$ is orientable.
\end{remark}

\begin{corollary}\label{ho}
The singular homology groups of the orbifold $ X $ with coefficients in $\QQ$ is 
$$H_k(X, \QQ) = \left\{ \begin{array}{ll} \QQ & \mbox{if} ~ k = 0 ~ \mbox{and if} ~ k = n ~ \mbox{even}\\
\oplus_{h_k}~\QQ & \mbox{if} ~ k ~ \mbox{is even},~ 0 < k < n\\
 0 & \mbox{otherwise}
\end{array} \right.
$$
\end{corollary}

With coefficients in $\ZZ_2$ the cellular chain complex \ref{ccc} is
\begin{equation}
\begin{CD}
0 \to \ZZ_2 @>0>> \oplus_{|I_{Q^{n-1}}|} ~\ZZ_2 @>0>> \cdots @>0>> \oplus_{|I_{Q^1}|}~\ZZ_2
@>d_1>> \oplus_{|I_{Q^0}|}~\ZZ_2 @>0>> 0 \\
\end{CD}
\end{equation}
Where $d_1$ is an injection. Hence $ (\oplus_{|I_{Q^0}|} ~\ZZ_2)/ d_1( \oplus_{|I_{Q^1}|}~\ZZ_2 ) $ is isomorphic
to $ \ZZ_2 $. So we get the following corollary.
\begin{corollary}
The singular homology groups of the orbifold $ X $ with coefficients in $\ZZ_2$ is
$$H_k(X, \ZZ_2) = \left\{ \begin{array}{ll} \ZZ_2 & \mbox{if} ~~ k = 0 ~ \mbox{and if} ~ k = n \\
\oplus_{\Sigma_{k}^{n} h_i}~\ZZ_2 & \mbox{if} ~ 1  < k < n\\
 0 & \mbox{if}~ k=1
\end{array} \right.
$$
\end{corollary}
\begin{remark}
The $k$-th modulo $2$ Betti number $b_k (X)$ of small orbifold $X$ is zero when $k=1$.
$b_k (X) =\Sigma_{k}^{n}~ h_i $ if $1 < k \leq n $ and $b_0 (X) = h_0 = 1 $.
Hence modulo $2$ Euler characteristic of $X$ is
\begin{equation}
\mathfrak{X}(X) = h_0 + \Sigma_{k=2}^{n}{(-1)^k \Sigma_{k}^{n} h_i} =\Sigma_{0}^{[n/2]} h_{2i}.
\end{equation}
Observe that  $b_k (X) \neq  b_{n-k} (X)$ if $1 \leqslant k < n$. Hence the Poincar\'{e} Duality 
for small orbifolds is not true with coefficients in $\ZZ_2$.
\end{remark}

\section{Cohomology ring of small orbifolds}\label{crto}
We have shown that the even dimensional small orbifolds are compact, connected, orientable.
Let $\mathcal{X}$ be an even dimensional small orbifold over the polytope $P$. 
Hence by the following Proposition we get that the cohomology ring of
$\mathcal{X}$ satisfy the Poincar\'{e} duality with coefficients in rationals.
\begin{prop}[Proposition $1.28$, \cite{ALR}]
If a compact, connected Lie group $G$ acts smoothly and almost freely on an orientable,
connected, compact manifold $M$, then the cohomology ring
$H^{\ast}(M/G; \QQ)$ is a Poincar\'{e} duality algebra.
Hence, if $\mathcal{X}$ is a compact, connected, orientable orbifold, then $H^{\ast}(X; \QQ)$ will satisfy Poincar\'{e} duality.
\end{prop}

We rewrite Poincar\'{e} duality for small orbifolds using the intersection theory. The purpose is to show
the cup product in cohomology ring is Poincar\'{e} dual to intersection, see equation \ref{cupdu}.
The proof is akin to the proof of Poincar\'{e} duality for oriented
closed manifolds proved in \cite{GH}. To show these we construct a $\bf{q}$-$CW$ complex structure on $X$.
The $\bf{q}$-$CW$ complex structure on a Hausdorff topological space is constructed in \cite{PS}.

An open cell of $\bf{q}$-$CW$ complex is the quotient of an open ball by linear, orientation preserving action
of a finite group. Such an action preserves the boundary of open ball. The construction mirrors the
construction of usual $CW$ complex given in Hatcher \cite{Ha}. In \cite{PS} the authors show that
$\bf{q}$-cellular homology of a $\bf{q}$-$CW$ complex is isomorphic to its singular homology
with coefficients in rationals. Similarly we can show that $\bf{q}$-cellular cohomology of a
$\bf{q}$-$CW$ complex is isomorphic to its singular cohomology with coefficients in rationals.

Let $P$ be an $n$-dimensional simple polytope where $n$ is even and $\pi : X \to P$ be a small orbifold over $P$.
Let $P^{\prime}$ be the second barycentric subdivision of the polytope $ P $.
Let
\begin{equation}
 \{\eta_{\alpha}^k : \alpha \in \varLambda(k) ~\mbox{and}~ k = 0,1,\ldots, n\}
\end{equation}
be the simplices in  $P^{\prime}$. Here $k$ is the dimension of $ \eta_{\alpha}^k$ and $\varLambda(k)$ is an index set.
Let $(\eta_{\alpha}^k)^0$ be the relative interior of $k$-dimensional simplex $\eta_{\alpha}^k$.
\begin{defn}
A subset $Y \subseteq X$ is said to be relatively open subset of dimension $k$ if for each point $ y \in Y$ there exist
an orbifold chart $(\widetilde{U}, G, \psi)$ such that $\psi(V) \ni y$ is an open subset of $Y$, for some $k$-dimensional
submanifold $V$ of $\widetilde{U}$.
\end{defn}

Then $(\pi^{-1})(\eta_{\alpha}^k)^0$ is disjoint union of the following relatively open subsets
$$ \{(\sigma_{\alpha_i}^k)^0 \subset X : i= 1,\ldots, \alpha(k)\} $$
for some natural number $\alpha(k)$. Here $\sigma_{\alpha_i}^k$ is the closure of $ (\sigma_{\alpha_i}^k)^0 $ in $X$.
The restriction of $\pi$ on $\sigma_{\alpha_i}^k$ is a homeomorphism onto the simplex $ \eta_{\alpha}^k $
for $ i= 1,\ldots, \alpha(k) $. Then the collection 
\begin{equation}
\{\sigma_{\alpha_i}^k : i= 1,\ldots, \alpha(k) ~\mbox{and}~ \alpha \in \varLambda(k) ~\mbox{and}~ k = 0,1,\ldots, n\} 
\end{equation}
gives a simplicial decomposition of the small orbifold $X$. So 
\begin{equation}
 \mathcal K = \{\sigma_{\alpha_i}^k, \partial\}_{\alpha_i, k}
\end{equation}
is a simplicial complex of $ X $.
\begin{defn}
The transversality of two relatively open subsets $U$ and $V$ of $X$ at $p \in U \cap V$ is defined as follows:
\begin{enumerate}
\item If $p$ is a smooth point of $X$, we say $ U $ intersect $ V $ transversely at $p$ whenever $T_{p}(U) + T_{p}(V) = T_{p}(X)$.

\item If $p$ is an orbifold point of $X$ there exist an orbifold chart $(B^n, \ZZ_2, \phi_v)$ such that $ \phi_v(0)=p$.
We say $U$ intersect $V$ transversely at $p$ whenever $T_{0}(\phi_v^{-1}(U)) + T_{0}(\phi_v^{-1}(V)) = T_{0}(B^n)$.
\end{enumerate}
\end{defn}

Let $\sigma_{\alpha_i}^{k_1} $ and $\rho_{\beta_j}^{k_2}$ be two simplices of dimension $k_1$ and $k_2$
respectively in the simplicial complex $\mathcal{K}$ of $X$.
\begin{defn}
We say $\sigma_{\alpha_i}^{k_1} $ and $\rho_{\beta_j}^{k_2}$ intersect transversely at
$p\in \sigma_{\alpha_i}^{k_1} \cap \rho_{\beta_j}^{k_2}$ if there exist two relatively open subsets
$U \subset X$ and $V \subset X$ containing $\sigma_{\alpha_i}^{k_1} $ and $\rho_{\beta_j}^{k_2}$
respectively such that dim$(U)= k_1$, dim$(V)= k_2$ and $U $ intersect $V$ transversely at $p$.
\end{defn}

Let $U$ and $V$ be two complementary dimensional relatively open subset of $X$ that intersect transversely at $p \in U \cap V$.
\begin{defn}
Define the intersection index of $U$ and $V$ at $p$ to be $1$ if there exist oriented bases $\{\xi_1, \ldots, \xi_{k_1}\}$ and
$\{ \eta_1, \ldots, \eta_{k_2}\}$ for $T_{p}(U)$ $ (T_{0}(\phi_v^{-1}(U)))$ and  $T_{p}(V)$ $(T_{0}(\phi_v^{-1}(V)))$ respectively
such that $\{\xi_1, \ldots, \xi_{k_1}, \eta_1, \ldots, \eta_{k_2}\}$ is an
oriented basis for $T_pX (T_0 B^n)$ whenever $p$ is smooth (respectively orbifold) point of $X$.
Otherwise the intersection index of $U$ and $V$ at $p$ is $-1$.
\end{defn}
Since antipodal action on $B^n$ (as $n$ is even) is orientation preserving there is no ambiguity in the above definition.
Let $$A = \Sigma n_{\alpha_i} \sigma_{\alpha_i}^{k_1} ~\mbox{and}~  B = \Sigma m_{\beta_j} \rho_{\beta_j}^{k_2}$$
be two cycles of the simplicial complex $\mathcal{K}$ such that $n=k_1+k_2$ and they intersect transversely.
\begin{defn}
Define the intersection number of $A$ and $B$ is the sum of the intersection indixes (counted with multiplicity)
at their intersection points.
\end{defn}
The number is finite since $A$ and $B$ are closed subsets of compact space $X$.
We show that the intersection number depends only on the homology class of the cycle.
Let $\sigma_{\alpha_i}^{k_1} $ and  $\rho_{\beta_j}^{k_2}$ be two simplices in $\mathcal{K}$ with $k_1 + k_2 = n$.
From the construction of the simplicial complex $\mathcal{K}$ we make some observations.
\begin{obser}\label{obssigro}
 \begin{enumerate}
 \item $\sigma_{\alpha_i}^{k_1} $ and $\rho_{\beta_j}^{k_2}$ can not contain different orbifold points whenever their 
intersection is nonempty.
\item Each $\sigma_{\alpha_i}^{k_1} $ and $\rho_{\beta_j}^{k_2}$ can contain atmost one orbifold point.

\item If $\sigma_{\alpha_i}^{k_1} $ and  $\rho_{\beta_j}^{k_2}$ contain an orbifold point or not, whenever their 
intersection is nonempty, we can find a $\ZZ_2$-invariant smooth homotopy $$\mathcal{G} : [0,1] \times X_v \to X_v$$
fixing the orbifold point of $X_v$ such that $\mathcal{G}(0 \times U_{\alpha_i}^{k_1})$ and 
$\mathcal{G}(1 \times V_{\beta_j}^{k_2})$ intersect
transversely where $U_{\alpha_i}^{k_1} $ and $ V_{\beta_j}^{k_2} $ containing $\sigma_{\alpha_i}^{k_1} $ and
 $\rho_{\beta_j}^{k_2}$ respectively are suitable relatively open subsets of $X_v$ and dim$U_{\alpha_i}^{k_1}=k_1$,
dim$ V_{\beta_j}^{k_2} = k_2$.
\end{enumerate}
\end{obser}
Let $\sigma_{\alpha_0}^{k_1} + \ldots + \sigma_{\alpha_{k_1}}^{k_1} $ be the boundary of $(k_1+1)$-simplex $\sigma_{\alpha}^{k_1+1}$.
The observations \ref{obssigro} also hold for the simplices $\sigma_{\alpha}^{k_1+1} $ and $\rho_{\beta_j}^{k_2}$
although $k_1+1+k_2 = n+1$. If $\mathcal{G}^{\prime}$ is the smooth homotopy and
 $\mathcal{G}^{\prime}(0 \times U_{\alpha}^{k_1+1}) \cap \mathcal{G}^{\prime}(1 \times V_{\beta_j}^{k_2})$ is nonempty then the subset
$$\mathcal{G}^{\prime}(0 \times U_{\alpha}^{k_1+1}) \cap \mathcal{G}^{\prime}(1 \times V_{\beta_j}^{k_2})$$ of $X$
is a collection of piecewise smooth curves. After lifting a curve to an orbifold chart (if necessary),
using the similar arguments as in \cite{GH} we can show that intersection number of
$\sigma_{\alpha_0}^{k_1} + \ldots + \sigma_{\alpha_{k_1}}^{k_1} $ and $\rho_{\beta_j}^{k_2}$ is zero.
Integrating these computation to the boundary $A = \Sigma n_{\alpha_i} \sigma_{\alpha_i}^{k_1} $ and the cycle
$ B = \Sigma m_{\beta_j} \rho_{\beta_j}^{k_2}$ we ensure that the intersection number of $A$ and $B$ is zero.

Let $\mathcal{K}^{\prime} = \{\tau_{\alpha_i}^k, \partial\} $ be the first barycentric subdivision of the complex $\mathcal{K}$.
Now we construct the dual $\bf{q}$-cell decomposition of the complex $\mathcal{K}$.
For each vertex $\sigma_{\alpha_i}^0$ in the complex $\mathcal{K}$, let
\begin{equation}
\ast\sigma_{\alpha_i}^0 = \displaystyle\bigcup_{_{\sigma_{\alpha_i}^0 \in \tau_{\beta_j}^{n}}} \tau_{\beta_j}^{n}
\end{equation}
be the $n$-dimensional $\bf{q}$-cell which is the union of the $n$-simplices $\tau_{\beta_j}^{n} \in \mathcal{K}^{\prime}$
containing $\sigma_{\alpha_i}^0$ as a vertex.
Then for each $k$-simplex $\sigma_{\alpha_i}^k$ in the decomposition $\mathcal{K}$, let
\begin{equation}
\ast\sigma_{\alpha_i}^k = \displaystyle\bigcap_{_{\sigma_{\beta_j}^0 \in \tau_{\alpha_i}^{n}}} \ast\sigma_{\beta_j}^0
\end{equation}
be the intersection of the $n$-dimensional $\bf{q}$-cells  associated to the $k+1$ vertices of $\sigma_{\alpha_i}^k$.
The $\bf{q}$-cells $\{\Delta_{\alpha_i}^{n-k} = \ast\sigma_{\alpha_i}^k \}$ give a $\bf{q}$-cell
decomposition of $X$, called the dual $\bf{q}$-cell decomposition of $\mathcal{K}$.
So the dual $\bf{q}$-cell decomposition $\{\Delta_{\alpha}^{n-k}\}$ is a $\bf{q}$-$CW$ structure on $X$.

From the description of dual $\bf{q}$-cells it is clear that $\Delta_{\alpha_i}^{n-k}$ intersects
$\sigma_{\alpha_i}^{k}$ transversely when dimension of $\sigma_{\alpha_i}^{k}$ is
greater than zero. $\Delta_{\alpha_i}^{n}$ is a quotient space of the antipodal action
on a symmetric convex polyhedral centred at origin in $\RR^{n}$.
Since the antipodal action on $\RR^{n}$ ($n$ even) preserve orientation of $\RR^{n}$,
we can define the intersection number of $\sigma_{\alpha_i}^{0}$ and $\Delta_{\alpha_i}^{n}$ to be $1$.
We consider the orientation on the dual $\bf{q}$-cell $\{\Delta_{\alpha_i}^{n}\}$ such that
the intersection number of $\sigma_{\alpha_i}^{k}$ and $\Delta_{\alpha_i}^{n-k}$ is $1$.

Using the same argument as Grifiths and Harris have made in the proof of Poincar\'{e} duality theorem in \cite{GH},
we can prove the following relation between boundary operator $\partial$ on the cell complex
$\{\sigma_{\alpha_i}^{k}\}$ and coboundary operator $\delta$ on the dual $\bf{q}$-cell complex
$\{\Delta_{\alpha_i}^{n-k}\}$ when dimension of $\sigma_{\alpha_i}^{k}$ is greater than one,
\begin{equation}
\delta(\{\Delta_{\alpha_i}^{n-k}\}) = (-1)^{n-k+1}\ast (\partial \sigma_{\alpha_i}^k).
\end{equation}

Let $\{\sigma_{\alpha_i}^{k}\} = <x,y> \in \mathcal{K}$ be a one simplex with the vertices $x, y$.
The orientation on $\{\sigma_{\alpha_i}^{k}\} $ comes from the orientation of $X$.
Since we are considering $\bf{q}$-cell structure on $X$, define
$\delta(\{\Delta_{\alpha_i}^{n-1}\}) = \ast \sigma_{y}^0 -  \ast \sigma_{x}^0$.
So we get a map $ \sigma_{\alpha_i}^k \to \Delta_{\alpha_i}^{n-k}$ which induces an isomorphism
\begin{equation}
\xi_k^{\prime} : H_{k}(X,\QQ) \to H^{n-k}_{\bf{q} \mbox{-}CW}(X,\QQ),
\end{equation}
where $H^{n-k}_{\bf{q} \mbox{-}CW}(X,\QQ)$ is $n-k$ th $\bf{q}$-cellular cohomology group.
Hence we have the following theorem for even dimensional small orbifold.

\begin{theorem}[Poincar\'{e} duality]
Let $X$ be an even dimensional small orbifold. The intersection pairing
 $$H_k(X,\QQ) \times H_{n-k}(X,\QQ) \to \QQ$$
is nonsingular; that is, any linear functional $H_{n-k}(X,\QQ) \to \QQ $ is expressible as the
intersection with some class $\Theta \in H_{k}(X,\QQ)$. There is an isomorphism $\xi_k^{\prime}$
from $ H_{k}(X,\QQ) $ to $H^{n-k}(X,\QQ)$.
\end{theorem}

Using this Poincar\'{e} duality theorem for even dimensional small orbifold
we can calculate the cohomology groups of small orbifold $X$.
\begin{theorem}\label{cohom}
The singular cohomology groups of the even dimensional small orbifold $ X $ with coefficients in $\QQ$ is
$$
H^k(X) = \left\{ \begin{array}{ll} \QQ & \mbox{if} ~ k = 0 ~ \mbox{and if} ~ k = n ~ \mbox{even}\\
\oplus_{h_k}~\QQ & \mbox{if} ~ k ~ \mbox{is even},~ 0 < k < n\\
0 & \mbox{otherwise}.
\end{array} \right.
$$
\end{theorem}
We can also define a product $\mu_{k_1k_2}$ similarly as in \cite{GH} but some care is needed at orbifold points.
The product
\begin{equation}
\mu_{k_1k_2} : H_{n-k_1}(X,\QQ) \times H_{n-k_2}(X,\QQ) \to H_{n-k_1-k_2}(X,\QQ)
\end{equation}
on the homology of $X$ in arbitrary dimensions satisfying the following commutative diagram.
\begin{equation}\label{cupdu}
\begin{CD}
H_{n-k_1}(X,\QQ) \times H_{n-k_2}(X,\QQ) @>\mu_{k_1k_2}>> H_{n-k_1-k_2}(X,\QQ)\\
@V\xi_{n-k_1}\times \xi_{n-k_2}VV  @V\xi_{n-k_1-k_2}VV \\
H^{k_1}(X,\QQ) \times H^{k_2}(X,\QQ) @>\mathfrak{u} >> H^{k_1+k_2}(X,\QQ)
\end{CD}
\end{equation}
where the lower horizontal map $\mathfrak{u}$ is the cup product in cohomology ring.

We write some observations about the transversality of faces of an $n$-dimensional polytope $P$ ($n$ even).
Let $F$ and $F^{\prime}$ be two faces of $P$. $F$ and $F^{\prime}$ intersect transversely if
codim$(F \cap F^{\prime})$ $=$ codim$ F~ +$ codim$ F^{\prime}$.
Since $P$ is simple polytope, the following two properties are satisfied.
\begin{property}\label{propo1}
 Let $F$ be a $2k$-dimensional face of $P$ and $u$ be a vertex of $F$. Then there is a unique $(n -2k)$-dimensional
face $F^{\prime}$ of $P$ such that $F$ and $F^{\prime}$ meet at $u$ transversely.
\end{property}
\begin{property}\label{propo2}
 Let $F$ be a face of codimension $2k$. Then there is $k$ many distinct faces of codimension two such that
they intersect transversely at each point of $F$.
\end{property}
%\end{equation}
%Following Lemma is a simple consequences of Poincar\'{e} duality and transversality.
\begin{lemma}
Let $\pi: X \to P$ be an even dimensional small orbifold and $X(F, \vartheta^{\prime}) = \pi^{-1}(F)$ for
each face $F$ of $P$. Then
\begin{enumerate}
\item For each $2k$-dimensional face $F$ of $P$, the homology class represented by $X(F, \vartheta^{\prime})$, denoted by
$[X(F, \vartheta^{\prime})]$, is not zero in $H_{\ast}(X, \QQ)$.

\item The cohomology ring $H^{\ast}(X, \QQ)$ is generated by $2$-dimensional classes.
\end{enumerate}
\end{lemma}
\begin{proof}
The space $X(F, \vartheta^{\prime})$ is a $2k$-dimensional suborbifold of $X$, for each $2k$-dimensional
($0 \leq 2k \leq n$) face $F$ of $P$. By Corollary \ref{ho} we get that the homology in degree $2k$ of
$X$ is generated by the classes of form $[X(F, \vartheta^{\prime})]$, where $F$ is a $2k$-dimensional face.

By equation \ref{cupdu}, the dual of $X(F \cap F^{\prime}, \vartheta^{\prime})$ is the cup product of the dual of
$[X(F, \vartheta^{\prime})]$ and the dual of $[X(F^{\prime}, \vartheta^{\prime})]$, if $F$ and $F^{\prime}$
intersect transversely and otherwise the dual of $X(F \cap F^{\prime}, \vartheta^{\prime})$ is zero.

The property \ref{propo1} tells that there is an $(n-2k)$-dimensional face $F^{\prime}$ which intersects
$F$ transversely at a vertex of $P$. Since the homology classes $[X(F, \vartheta^{\prime})]$ and
$[X(F^{\prime}, \vartheta^{\prime})]$ are dual in intersection pairing of Poincar\'{e} duality, they are both nonzero.
This proves (1) of the above Lemma.

In theorem \ref{cohom} we show the odd dimensional cohomology ring is zero. The cohomology in degree $2k$ is
generated by Poincar\'{e} duals of classes of the form $[X(F, \vartheta^{\prime})]$, codim$ F = 2k$.
By property \ref{propo2}, $F$ is the transverse
intersection of distinct faces of codimension two. Hence, the Poincar\'{e} dual of $[X(F, \vartheta^{\prime})]$
is the product of cohomology classes of dimension $2$. This proves (2) of the above Lemma.
\end{proof}

Recall the index function $ind_P$ from section \ref{hom}. Let $\hat{F}_{v} \in \mathfrak{F}(P)$ be the smallest
face containing the inward pointing edges incident to the vertex $v$ of $P$. Let $w$ be the Poincar\'{e} dual
of class of the form $[X(\hat{F}_{v}, \vartheta^{\prime})]$, also denoted by $[v]$. Let $\{ v_1, v_2, \ldots,$ $ v_r\}$ be the set of
vertices of $P$ such that $ ind_P(v_{i})= n-2$. We show that $\{ w_1, w_2, \ldots, w_r\}$ is a minimal
generating set of $H^{\ast}(X, \QQ)$.

Let  $A_j= \{v \in V(P) : ind_P(v)= j $. Let $U_{\hat{F}_v}$ be the open subset of $\hat{F}_v$ obtained by
deleting all faces of $\hat{F}_v$ not containing the vertex $v$. From section \ref{def} it is clear that
$\pi^{-1}(U_{\hat{F}_v})$ is homeomorphic to the orbit space $B^j/ \ZZ_2$, where $\ZZ_2$ acttion on $B^j$
is antipodal. So $\pi^{-1}(U_{\hat{F}_v})$ is $j$-dimensional $\bf{q}$-cell in $X$. Clearly
$$ X = \bigcup_{v \in V(P)} \pi^{-1}(U_{\hat{F}_v}). $$
This gives a $\bf{q}$-$CW$ structure on $X$. From Theorem 1.20 of \cite{BP}, we get the number of
$j$-dimensional cells is $h_{n-j}$, cardinality of $A_j$. So the corresponding $\bf{q}$-cellular chain
complex gives that $\{[v] : v \in A_j\}$ is a basis of $H_{j}(X, \QQ)$ if $j$ is even.
Theorem \ref{cohom} tells that $\{ w = \xi_j([v]) : v \in A_j\}$ is a basis of $H^{n-j}(X, \QQ)$ if $j$ is even.

Let $F$ be a codimension $2k$ face of $P$ with top vertex $v$ of index $n-2k$. By property \ref{propo2}
$F$ is unique intersection of $k$ many distinct codimension $2$ faces $\hat{F}_{v_{i_1}}, \ldots, \hat{F}_{v_{i_k}}$
with top vertices $v_{i_1}, \ldots, v_{i_k} \in \{ v_1, v_2, \ldots,$ $ v_r\}$ respectively. 
Hence $w_{i_1} \ldots w_{i_k} = w$ in $H^{\ast}(X, \QQ)$. Consider the polynomial ring $\QQ[w_1, w_2, \ldots, w_r]$.
Let the map
\begin{equation}
\mu_{n_{i_1} \ldots n_{i_l}} : H_{n_{i_1}}(X,\QQ)\times \cdots \times H_{n_{i_l}}(X,\QQ) \to H_{n-n_{i_1}-\cdots -n_{i_l}}(X,\QQ)
\end{equation}
be defined by the repeated application of the product map $\mu_{n_{i_1}n_{i_2}}$.
Let $I$ be the ideal of $\QQ[w_1, w_2, \ldots, w_r]$  generated by the following elements
 
\begin{equation}
S = \left\{ \begin{array}{ll} w_{i_1}w_{i_2}\ldots w_{i_l} & \mbox{if}~ \mu_{n_{i_1} \ldots n_{i_l}}([v_{i_1}], \ldots, [v_{i_l}])=0
~\mbox{in} ~ H_{n-\{n_{i_1}+\cdots + n_{i_l}\}}(X, \QQ)\\
\prod_{1}^{l_1}w_{i_k} - \prod_{1}^{l_2}w_{j_l} & \mbox{if} ~
\mu_{n_{i_1} \ldots n_{i_{l_1}}}([v_{i_1}], \ldots ,[v_{i_{l_1}}])=\mu_{n_{j_1} \ldots n_{j_{l_2}}}([v_{j_1}], \ldots, [v_{j_{l_2}}])
~ \mbox{in}\\
& H_{n-\{n_{i_1}+\cdots + n_{i_l}\}}(X, \QQ) ~\mbox{with} ~ n_{i_1}+ \ldots + n_{i_{l_1}} =  n_{j_1} + \ldots + n_{j_{l_2}}
\end{array} \right.
\end{equation}
The Poincar\'{e} Duality theorem and intersection theory ensure that the relations among $w_i$'s are exactly as described
above. Hence we have the following theorem.
\begin{theorem}
The cohomology ring of even dimensional small orbifold $X$ over the simple polytope $P$ is isomorphic to the
quotient ring $ \QQ[w_1, w_2, \ldots, w_r]/I. $
\end{theorem}

\section{Some remarks on toric version}
\begin{defn}
 The function $\psi : \mathcal{F}(P) \to \ZZ^{n-1} $ is called an isotropy function of $P$
if the facets $F_{i_{1}}, F_{i_{2}}, \ldots , F_{i_{n}} $ intersect at vertex of $P$ then the set
$$\{ \psi_{i_{1}}, \psi_{i_{2}}, \ldots, \psi_{i_{k-1}}, \hat \psi_{i_{k}}, \psi_{i_{k+1}}, \ldots, \psi_{i_{n}} \},$$
where $ \psi(F_i) = \psi_i$, is a basis of $ \ZZ^{n-1} $ over $ \ZZ $ for each $ k $ ($ 1 < k < n $).
\end{defn}
Here the symbol $\hat{}$ represents the omission of corresponding entry.

The quotient  $\TT^{n-1} = (\ZZ^{n-1}\otimes \RR) / \ZZ^{n-1} $ is a compact $(n-1)$-dimensional torus.
Suppose $ F = F_1 \cap \ldots \cap F_l $. Let $G_F$ be the subgroup of $\TT^{n-1}$ determined by 
the span of $\psi_{1}, \ldots, ,\psi_{l} $. Let $S(P,\psi)$ be the quotient space of equivalence relation
$ \sim_T $ on $\TT^{n-1} \times P $ define by 
\begin{equation}
(t,p) \sim_T (s,q) ~ \mbox{if} ~ p = q ~ \mbox{and} ~ s^{-1}t \in G_{F(p)}
\end{equation}
where $F(p)$ is the unique face of polytope $P$ whose relative interior contains $p$.
Then every point of $S(P,\psi)$ are smooth point except a finite set of points corresponding to the set $V(P)$ if $n \geq 3$.
When $ n = 2 $ the quotient space is homeomorphic to $3$-sphere. 
Only this is the case where the quotient space is a manifold.

We can give a $CW$-structure on $S(P,\psi)$ with cells in dimension $0, 1, 3, \ldots, 2n-1$ only.
The zero dimensional cells correspond to the set $V(P)$.
The one dimensional cells correspond to the relative interior of each edge of a maximal tree of the $1$-skleton of $P$.
Hence in the cellular chain complex of the $CW$-structure on $S(P,\psi)$ each boundary map
$d_k^{\prime}$ is zero except $d_1^{\prime}$. The map $d_1^{\prime}$ is injective and the image
of the map $d_1^{\prime}$ is a direct summand of a free module with codimension-$1$.
Hence we can prove the following theorem.
\begin{theorem}
The singular homology of the space $ S(P,\psi) $ with $\ZZ$ coefficients is
$$H_k(S(P, \psi), \ZZ) = \left\{ \begin{array}{ll} \ZZ & \mbox{if} ~ k = 0 ~ \mbox{and if} ~ k = 2n-1\\
\oplus_{\Sigma_{l}^{n} h_i}~\ZZ & \mbox{if} ~ k = 2l-1 ~ \mbox{is odd and}~ 1 < k <  2n-1\\
0 & \mbox{otherwise}
\end{array} \right.
$$
\end{theorem}

\begin{remark}
If $n\geq 3$ and $P$ is a simple $n$-polytope, the space $ S(P,\psi) $ is not an orbifold. The Euler characteristic
of the space $S(P,\psi)$ is 
\begin{equation}
\mathfrak{X}(S(P,\psi), \ZZ) = h_0 + \displaystyle\sum_{k=2}^{n}{(-1)^{2k-1} \displaystyle\sum_{k}^{n} h_i}
=h_0 - \displaystyle\sum_{2}^{n} (i-1)h_{i}
\end{equation}
\end{remark}

\[\]
{\bf Acknowledgment.} The author would like to thank his advisor Mainak Poddar for helpful suggestions and
stimulating discussions. The author is thankful to the anonymous referee for helpful suggestions.
He would also like to thank I. S. I. for supporting his research grant during the preparation of the work was done.

\renewcommand{\refname}{References}
%\bibliographystyle{alpha}
%\bibliography{bibliography.bib}

%\vspace{1cm}
%\vfill
%\end{document}

\vspace{1cm}

\vfill
\end{document}